\documentclass[11pt]{article}

\usepackage[margin=1in]{geometry}  % set the margins to 1in on all sides
\usepackage{amsmath}               % great math stuff
\usepackage{amsfonts}              % for blackboard bold, etc
\usepackage{amsthm}                % better theorem environments

\newtheorem{theorem}{Theorem}[section]
\newtheorem{lemma}[theorem]{Lemma}
\newtheorem{proposition}[theorem]{Proposition}
\newtheorem{corollary}[theorem]{Corollary}

\newtheorem{definition}{Definition}

\newcommand{\R}{\mathbb{R}}
\newcommand{\E}{\mathbf{E}}

\providecommand{\keywords}[1]{\textit{Keywords:} #1}
\providecommand{\msc}[1]{\textit{AMS MSC 2010:} #1}

\begin{document}

\nocite{*}

\title{Example of a Gaussian self-similar field with stationary rectangular increments that is not a fractional Brownian sheet}
%On the probability of exceeding some level by self-similar Gaussian random fields with rectangular increments}

\author{Vitalii  Makogin \\
Department of Probability Theory, Statistics and Actuarial
Mathematics \\
Taras Shevchenko National University of Kyiv \\
Volodymyrska 64, Kyiv 01601, Ukraine, E-mail:makoginv@ukr.net
\and
 Yuliya Mishura \\
Department of Probability Theory, Statistics and Actuarial
Mathematics \\
Taras Shevchenko National University of Kyiv \\
Volodymyrska 64, Kyiv 01601, Ukraine , E-mail:myus@univ.kiev.ua
 }

\maketitle

\begin{abstract}
We consider anisotropic self-similar random fields, in particular, the fractional Brownian sheet. This Gaussian field is an extension of fractional Brownian motion. We prove some properties of covariance function for self-similar fields with rectangular increments. Using Lamperti transformation we obtain properties of covariance function for the corresponding stationary fields. We present an example of a Gaussian self-similar field with stationary rectangular increments that is not a fractional Brownian sheet.
\end{abstract}


\msc{60G60;60G18;60G22 (42A82).}

\section{Introduction}
A real valued random process $\{X(t),t\in \R_+\}, (\R_+=[0,+\infty))$ is called a self-similar process with index $H>0$ if for all $a>0$ $\{X(at),t\in \R_+\}\stackrel{d}{=}\{a^H X(t),t\in \R_+\},$ 
where symbol $\stackrel{d}{=}$ means that the corresponding finite-dimensional distributions coincide. The books by Embrechts \& Maejima \cite{maej} and
Samorodnitsky \& Taqqu \cite{Sam} are devoted to the theory of
self-similar processes.

Square integrable self-similar processes with stationary
increments have very precise form of covariance function (\cite{taqqu81}). Indeed, assume $\E[X(1)]^2<+\infty,$ then
\begin{align}
\label{eq:proc}
\begin{gathered}
\E[X(t)X(s)]=\frac{1}{2}\E\left(X^2(t)+X^2(s)-(X(t)-X(s))^2\right)\\
=\frac{1}{2}\left(\E X^2(t)+\E X^2(s)-\E(X(t-s))^2\right)\\
=\frac{1}{2}\left(t^{2H}+s^{2H}-|t-s|^{2H}\right)\E[X(1)]^2,\quad t,s\in \R_+.
\end{gathered}
\end{align}
It is easy to check that a real valued self-similar random process $\{X(t),t\in \R_+\}$ with index $H>0$ is centered. 
So, all square integrable self-similar processes with stationary increments have the same covariance
function.

It is known that the distribution of a Gaussian process is
determined by its mean and covariance structure. Thus, the formula \eqref{eq:proc} determines a
unique Gaussian process.
\begin{definition}
Let $0<H<1$. A real-valued Gaussian process $\{B_H(t),
t\in \R_+\}$ is called fractional Brownian motion with Hurst index $H$ if $\E[B_H(t)]=0$ and
\[\E[B_H(t)B_H(s)]=\frac{1}{2}\left(t^{2H}+s^{2H}-|t-s|^{2H}\right)\E[B_H(1)]^2, t,s \in \R_+.\]
\end{definition}

It is known that a fractional Brownian motion $\{B_H(t),
t > 0\}$ is a self-similar process with stationary increments. So, fractional Brownian motion is unique in the sense that the class of all fractional Brownian motions coincides
with that of all Gaussian self-similar processes with stationary increments.

In this paper we consider self-similar random fields that are an extension of self-similar stochastic processes. More precisely, we deal with anisotropic self-similar random fields which means that their indexes of self-similarity are different for different coordinates.
%In this paper we work with an self-similar random fields that are anisotropic.
\begin{definition}
\label{taq} 
A real valued random field $\{X (\mathbf{t}), \mathbf{t} = (t_1, \dots , t_n) \in \R_+^n\}$ is
self-similar with index $\mathbf{H} = (H_1, \ldots ,H_n) \in (0,+\infty)^n$
if \[\left\{X (a_1t_1, \ldots , a_nt_n), \ \ \mathbf{t} \in \R_+^n\right\} \stackrel{d}{=}\left\{a_1^{H_1}\cdots a_n^{H_n} X (\mathbf{t}), \ \ \mathbf{t} \in \R_+^n \right\},\] for all $a_1 > 0, \ldots , a_n > 0.$
\end{definition}

Interest to anisotropic self-similar random fields is motivated by applications
coming from climatological and environmental sciences (see
\cite{env1,env2}). Several authors have proposed to apply such
random fields for modelling phenomena in spatial statistics,
stochastic hydrology and imaging processing (see
\cite{env3,env4,env5}).

\begin{definition}
\label{fbs} The normalized fractional Brownian sheet with
Hurst index\\ $\mathbf{H}=(H_1,\ldots,H_n), 0<H_i<1, i=\overline{1,n}$
is the centered Gaussian random field
$B_\mathbf{H}=\{B_\mathbf{H}(\mathbf{t}),\mathbf{t} \in \R_+^n\}$
with a covariance function
\[\E(B_\mathbf{H}(\mathbf{t})B_\mathbf{H} (\mathbf{s})) =
2^{-n}\prod_{i=1}^{n}
\left(|t_i|^{2H_i}+|s_i|^{2H_i}-|t_i-s_i|^{2H_i}\right),\quad
\mathbf{t},\mathbf{s} \in \R_+^n. \]
\end{definition}
This field is  self-similar with index $\mathbf{H}=(H_1,\ldots,H_n)$ by Definition~\ref{taq}. 

Further in the paper, we  assume that the fields satisfy the
Definition~\ref{taq}.  
Moreover, we shall consider only the case $n=2$ since switching to the parameter of the higher
dimension is rather technical.

\begin{definition}
Let $X=\{X(\mathbf{t}),\mathbf{t}\in \R^2\}$ be a self-similar field with index $\mathbf{H}=(H_1,H_2)\in (0,+\infty)^2.$
For any $\mathbf{u}=(u_1,u_2) \in \R^2$ and any $\mathbf{v}=(v_1,v_2)\in \R^2$ such that $v_1>u_1,$ $v_2>u_2$ define
\[ \Delta_{\mathbf{u}} X(\mathbf{v}) =  X(v_1,v_2)-X(u_1,v_2)-X(v_1,u_2)+X(u_1,u_2).
\]
The field $X$ admits stationary rectangular increments if for
any $\mathbf{u}=(u_1,u_1) \in \R^2$
\[\{\Delta_{\mathbf{u}} X(\mathbf{u}+\mathbf{h}),\mathbf{h}\in \R^2_+\}\stackrel{d}{=}\{\Delta_{0,0} X(\mathbf{h}),\mathbf{h}\in \R^2_+\}.\]
\end{definition}

The fractional Brownian sheet has stationary rectangular increments. The proof of this property for
the $\R^2$ case can be found in the paper \cite{alp}. A similar
property for the case  $n>2$ can be easily proved as well.

The properties of fractional Brownian sheet and fractional Brownian motion seem to be quite similar. The aim of this paper is an answer to the following question: 

\textit{Is
fractional Brownian sheet unique Gaussian self-similar fields with stationary rectangular increments?} 

The answer is no and we present an example of a Gaussian self-similar field with stationary rectangular increments that is not a fractional Brownian sheet.

Let us mention that the self-similar process has not to be stationary. But there is a one-to-one
correspondence between self-similar and stationary processes. For every self-similar process $X$ with index $H>0$, its Lamperti
transformation $Z= \{Z(t)=t^{-H}X(e^t)\}$ is a stationary process.
The Lamperti transformation for anisotropic random fields was introduced in the paper \cite{taqqu} and there was established
the correspondence between self-similar and stationary random
fields as well. We get necessary and sufficient conditions on covariance function of stationary field for the corresponding self-similar field to have have stationary rectangular increments.

The rest of the paper is organized as follows. 
In Section 2 we prove some properties of covariance function for self-similar fields with rectangular increments.
In Section 3, we present Lamperti transformation of self-similar field and obtain properties of covariance function for the corresponding stationary field. In Section 4 we present an example of a Gaussian self-similar field with stationary rectangular increments that is not a fractional Brownian sheet.

\section{Some properties of self-similar random fields}
Throughout  this section the field  $X=\{X(\mathbf{t}),\mathbf{t}\in\R^2_+\}$ is a self-similar random filed with index $\mathbf{H} =(H_1,H_2),
0<H_1<1,~0<H_2<1$ and with stationary rectangular increments. %Also
%assume that $\E X^2(1,1)=1$, then, obviously,
Evidently,  \[\E[X(\mathbf{t})]^2=t_1^{2H_1}t_2^{2H_2}\E[X^2(1,1)].\]
In what follows we need some auxiliary results.
\begin{lemma}
\label{lemma_1}
For all $\mathbf{s},\mathbf{t}\in\R^2_+$ we have
\begin{eqnarray}
\label{lemma_ineq_1}
  \E\left[X(\mathbf{t})-X(s_1,t_2)\right]^2 &=& |t_1-s_1|^{2
H_1}t_2^{2H_2}\E
X^2(1,1), \\
\label{lemma_ineq_2}
  \E\left[X(s_1,t_2)-X(\mathbf{s})\right]^2 &=& |t_2-s_2|^{2
H_2}s_1^{2H_2}\E X^2(1,1).
\end{eqnarray}
\end{lemma}
\begin{proof}
Without loss of generality suppose that $s_1\leq t_1.$ It follows from Proposition 2.4.1 of \cite{taqqu} that for any
$s>0:$ $X(s,0)=X(0,s)=0$ a.s. Then the left-hand  side of (\ref{lemma_ineq_1})  equals to
\[\E\left(X(\mathbf{t})-X(t_1,0)-X(s_1,t_2)+X(s_1,0)\right)^2=\E\left(\Delta_{s_1,0} X(\mathbf{t})\right)^2.\]
Stationarity of the increments implies that
\[\E\left(\Delta_{s_1,0} X(\mathbf{t})\right)^2=\E\left(\Delta_{0,0} X(t_1-s_1,t_2)\right)^2=\E\left(X(t_1-s_1,t_2)\right)^2.\]
Now,  self-similarity implies that
\[\E\left(X(\mathbf{t})-X(s_1,t_2)\right)^2=\E\left(X(t_1-s_1,t_2)\right)^2=|t_1-s_1|^{2
H_1}t_2^{2H_2}\E X^2(1,1).\] The proof of equality (\ref{lemma_ineq_2}) is
similar.
\end{proof}

\begin{lemma}
\label{lemma_2}
For all $\mathbf{s},\mathbf{t}\in\R^2_+$ we have
\begin{eqnarray}
\label{lemma_ineq_3}
  \E\left[X(\mathbf{t})X(s_1,t_2)\right] &= \frac{1}{2}t_2^{2H_2}\left(t_1^{2H_1}+s_1^{2H_1}-|t_1-s_1|^{2
H_1}\right)\E X^2(1,1), \\
\label{lemma_ineq_4}
  \E\left[X(s_1,t_2)X(\mathbf{s})\right] &= \frac{1}{2}s_1^{2H_1}\left(t_2^{2H_2}+s_2^{2H_2}-|t_2-s_2|^{2
H_2}\right)\E X^2(1,1).
\end{eqnarray}
\end{lemma}
\begin{proof}
It follows from Lemma \ref{lemma_1}   that
\begin{align*}
\E\left[X(\mathbf{t})X(s_1,t_2)\right] &= \frac{1}{2}\left(\E X^2(\mathbf{t})+\E X^2(s_1,t_2)-\E\left[X(\mathbf{t})-X(s_1,t_2)\right]^2\right)\\
&=\frac{1}{2}\left(t_1^{2H_1}t_2^{2H_2}+s_1^{2H_1}t_2^{2H_2}-|t_1-s_1|^{2
H_1}t_2^{2H_2}\right)\E X^2(1,1).
\end{align*}
Similarly  we obtain that
\begin{align*}
\E\left[X(s_1,t_2)X(\mathbf{s})\right] &= \frac{1}{2}\left(\E X^2(\mathbf{s})+\E X^2(s_1,t_2)-\E\left[X(s_1,t_2)-X(\mathbf{s})\right]^2\right)\\
&=\frac{1}{2}\left(s_1^{2H_1}s_2^{2H_2}+s_1^{2H_1}t_2^{2H_2}-|s_2-t_2|^{2
H_2}s_1^{2H_1}\right)\E X^2(1,1).
\end{align*}

\end{proof}

\begin{lemma}
\label{lemma_3}
For all $\mathbf{s},\mathbf{t}\in\R^2_+$ we have
\begin{align}
\label{lemma_ineq_5}
\begin{gathered}
  \E\left[X(\mathbf{t})X(\mathbf{s})\right]+\E\left[X(t_1,s_2)X(s_1,t_2)\right] \\
  =  \frac12 \prod_{i=1,2}\left(t_i^{2H_i}+s_i^{2H_i}-|t_i -
s_i|^{2H_i} \right)\E X^2(1,1).
\end{gathered}
\end{align}
\end{lemma}
\begin{proof}
Let $s_1\leq t_1,$ $s_2\leq t_2.$
By stationarity of increments, we have
\[\E\left(\Delta_{\mathbf{s}} X(\mathbf{t})\right)^2=\E\left(\Delta_{0,0} X(\mathbf{t}-\mathbf{s})\right)^2=\E X^2(\mathbf{t}-\mathbf{s}).\]
It follows from definition of rectangular increments that
\begin{equation*}\begin{gathered}
\E X^2(\mathbf{t}-\mathbf{s})=\E\left(\Delta_{\mathbf{s}} X(\mathbf{t})\right)^2=\E\left(X(\mathbf{t})-X(t_1,s_2)-X(s_1,t_2)+X(\mathbf{s})\right)^2\\
=\E\left[X(\mathbf{t})-X(t_1,s_2)\right]^2+\E\left[X(s_1,t_2)-X(\mathbf{s})\right]^2
+2\E\left[X(\mathbf{t})X(\mathbf{s})\right]\\+2\E\left[X(t_1,s_2)X(s_1,t_2)\right]
-2 \E\left[X(\mathbf{t})X(s_1,t_2)\right]-2 \E\left[X(\mathbf{s})X(t_1,s_2)\right].\end{gathered}\end{equation*}
From Lemmas \ref{lemma_1} and \ref{lemma_2} we immediately get  (\ref{lemma_ineq_5}).

In the case $s_1\geq t_1,$ $s_2\geq t_2$ the proof is similar, and in the case
 $s_1\geq t_1,$ $s_2\leq t_2$ we only replace $s_1$ with $t_1$ and vice versa. Lemma is proved.
%$$\E X^2(s_1-t_1,t_2-s_2)=\E\left(\Delta_{t_1,s_2} X(s_1,t_2)\right)^2=\E\left(X(s_1,t_2)-X(s_1,s_2)-X(t_1,t_2)+X(t_1,s_2)\right)^2=$$
%$$=\E\left(X(t_1,t_2)-X(t_1,s_2)-X(s_1,t_2)+X(s_1,s_2)\right)^2.$$
%Further proof coincides with the proof of the case $s_1\leq t_1,$ $s_2\leq t_2.$

\end{proof}

\section{Lamperti Transformation of Self-Similar Fields}
Let $X=\{X(\mathbf{t}),\mathbf{t}\in\R^2_+\}$ be a self-similar random filed with index $\mathbf{H}=(H_1,H_2),
0<H_1<1,~0<H_2<1.$
% with stationary rectangular increments.
Introduce the Lamperti representation of $X$ that has the form
\begin{equation}
\label{lamp}
X (\mathbf{t}) = t_1^{H_1}t_2^{H_2} Y \left(\ln t_1,\ln t_2 \right),\ \mathbf{t}\in \R^2_+,
\end{equation}
where $Y=\{Y(\mathbf{t}),\mathbf{t}\in \R^2\}$ is a new random field.
It follows from Proposition 2.1.1 of \cite{taqqu} that $Y$ is zero  mean strictly stationary field, i.e.
\[(Y(\mathbf{t}^1+\mathbf{h}),\ldots,Y(\mathbf{t}^n+\mathbf{h}))=(e^{-H_1 t_1^1-H_1 h_1}e^{-H_2 t_2^1-H_2 h_2}X(e^{t_1^1}e^{h_1},e^{t_2^1}e^{h_2}),\ldots,\]
\[e^{-H_1 t_1^n-H_1 h_1}e^{-H_2 t_2^n-H_2 h_2}X(e^{t_1^n}e^{h_1},e^{t_2^n}e^{h_2}))\stackrel{d}{=}(e^{-H_1 t_1^1}e^{-H_2 t_2^1}X(e^{t_1^1},e^{t_2^1}),\ldots,\]
\[e^{-H_1 t_1^n}e^{-H_2 t_2^n}X(e^{t_1^n},e^{t_2^n}))=(Y(\mathbf{t}^1),\ldots,Y(\mathbf{t}^n)).\]
Denote its  covariance function
\begin{equation}
\label{rr}
R(\mathbf{v})=\E[Y (\mathbf{t})Y (\mathbf{t}+\mathbf{v})],\ \mathbf{t},\mathbf{v} \in\R^2.
\end{equation}

Introduce the following notations. Let
\begin{equation}
\label{fh}
F_H(v)=e^{H v}+e^{-H v}-\left|e^{v/2}-e^{-v/2}\right|^{2 H}=2\cosh {(Hv)}-\left|2\sinh(v/2)\right|^{2H},\ v\in \R ,
\end{equation}
and
\[R_0(\mathbf{v})=\prod_{i=1,2}\left(\cosh(H_i v_i)-2^{(2 H_i -1)}\left|\sinh(v_i/2)\right|^{2 H_i}\right)=\frac{1}{4}\prod_{i=1,2}F_{H_i}(v_i),\ \mathbf{v} \in \R^2,\]
where $H, H_1,  H_2\in(0,1).$ Note that for fractional Brownian sheet $B_{H_1,H_2}$ the corresponding stationary filed has covariance function $R_0.$
From now on we assume that  $\E X^2(1,1)=1.$
\begin{proposition}
\label{prop1}
Let $X=\{X(\mathbf{t}), \mathbf{t}\in\R^2_+\}$ be a self-similar random field with index $\mathbf{H}=(H_1,H_2)$
and  $R$ be a covariance function (\ref{rr}) of a stationary field $Y$ in Lamperti transformation of X. If the field $X$ has stationary rectangular increments then
\begin{equation}
\label{r1}
R(\mathbf{v})+R(v_1,-v_2)=\frac{1}{2}F_{H_1}(v_1)F_{H_2}(v_2)=2R_0(\mathbf{v}), \mathbf{v}\in\R^2.
\end{equation}
\end{proposition}
\begin{proof}
It follows from the definition \eqref{lamp} of Lamperti transformation  that
\begin{equation}
\label{eq_2}
\begin{gathered}
\E\left[X (\mathbf{t})X (\mathbf{s})\right]=t_1^{H_1}s_1^{H_1}t_2^{H_2}t_2^{H_2}\E\left[Y \left(\ln t_1 ,\ln t_2 \right)Y \left(\ln s_1 ,\ln s_2 \right)\right]\\
=t_1^{H_1}s_1^{H_1}t_2^{H_2}t_2^{H_2}R\left(\ln \frac {t_1}{s_1},\ln \frac{t_2}{s_2}\right).
\end{gathered}
\end{equation}
So
\begin{equation}
\label{eq_3}
\begin{gathered}
\E\left[X (\mathbf{t})X (\mathbf{s})\right]+\E\left[X (s_1,t_2)X (t_1,s_2)\right]\\
=t_1^{H_1}s_1^{H_1}t_2^{H_2}t_2^{H_2}\left(\E\left[Y \left(\ln t_1 ,\ln t_2 \right)Y \left(\ln s_1 ,\ln s_2 \right)\right]+\E\left[Y \left(\ln s_1,\ln t_2 \right)Y \left(\ln t_1,\ln s_2 \right)\right]\right)\\
=t_1^{H_1}s_1^{H_1}t_2^{H_2}t_2^{H_2}\left(R\left(\ln t_1 -\ln s_1,\ln t_2 -\ln s_2\right)+R\left(\ln t_1 -\ln s_1,\ln s_2 -\ln t_2\right)\right).
\end{gathered}
\end{equation}

It follows immediately from Lemma \ref{lemma_3} and (\ref{eq_3}) that
\[R\left(\ln \frac {t_1}{s_1},\ln \frac{t_2}{s_2}\right)+R\left(\ln \frac{t_1}{s_1},-\ln \frac{t_2}{s_2}\right)=\frac12 \prod_{i=1,2}t_i^{-H_i}s_i^{-H_i}\left(|t_i|^{2H_i}+|s_i|^{2H_i}-|t_i -
s_i|^{2H_i} \right)\]
\[=\frac12 \prod_{i=1,2}\left(\left(\frac{t_i}{s_i}\right)^{H_i}+\left(\frac{s_i}{t_i}\right)^{H_i}-\left|\frac{t_i}{s_i} -
\frac{s_i}{t_i}\right|^{2H_i} \right).\]
Hence,
\[R(\mathbf{v})+R(v_1,-v_2)=\frac12 \prod_{i=1,2}\left(e^{H_i v_i}+e^{-H_i v_i}-\left|e^{v_i/2}-e^{-v_i/2}\right|^{2 H_i}\right)=\frac{1}{2}F_{H_1}(v_1)F_{H_2}(v_2).\]
\end{proof}
\begin{corollary}
$R(\mathbf{v})=R(v_1,-v_2)$ for all $\mathbf{v}\in\R^2$ if and only if $R=R_0.$
\end{corollary}
\begin{proposition}
\label{rem1}
Let $Y$ be a stationary field, whose covariance function (\ref{rr}) satisfies the equality (\ref{r1}).
Let $X$ be defined via (\ref{lamp}). Then $X$ is self-similar and
\[\E\left(\Delta_{\mathbf{s}} X(\mathbf{t})\right)^2=(t_1-s_1)^{2
H_1}(t_2-s_2)^{2
H_2}=\E\left(\Delta_{0,0} X(\mathbf{t}-\mathbf{s})\right)^2, 0\leq s_1\leq t_1,0\leq s_2\leq t_2.\]
\end{proposition}
\begin{proof}
Self-similarity of $X$ follows immediately from (\ref{lamp}).
From (\ref{r1}) we have that
\[R(0,v)=\frac{1}{4}F_{H_1}(0)F_{H_2}(v)=\frac{1}{2}F_{H_2}(v),\quad R(v,0)=\frac{1}{2}F_{H_1}(v).\]
Furthermore, we have the following evident equality
\begin{equation}
\label{tt_1}
\begin{gathered}
\E\left(\Delta_{\mathbf{s}} X(\mathbf{t})\right)^2=\E\left(X(\mathbf{t})-X(t_1,s_2)-X(s_1,t_2)+X(\mathbf{s})\right)^2\\
=\E X^2(\mathbf{t})+\E X^2(t_1,s_2)+\E X^2(s_1,t_2)+\E X^2(\mathbf{s})\\
-2 \E\left[X(\mathbf{t})X(t_1,s_2)\right]-2 \E\left[X(\mathbf{t})X(s_1,t_2)\right]
-2 \E\left[X(t_1,s_2)X(\mathbf{s})\right]\\-2 \E\left[X(s_1,t_2)X(\mathbf{s})\right]
+2\E\left[X(\mathbf{t})X(\mathbf{s})\right]+2\E\left[X(t_1,s_2)X(s_1,t_2)\right].
\end{gathered}
\end{equation}

Let $s_i>0$ (for $s_i=0$ proof is similar but more simple). It follows from (\ref{eq_2}) that the right- hand side of (\ref{tt_1}) equals to
\[t_1^{2H_1}t_2^{2H_2}+t_1^{2H_1}s_2^{2H_2}+s_1^{2H_1}t_2^{2H_2}+s_1^{2H_1}s_2^{2H_2}\]
\[-2t_1^{2H_1}t_2^{H_2}s_2^{H_2}R\left(0,\ln \frac{t_2}{s_2}\right)-2t_1^{H_1}s_1^{H_1}t_2^{2H_2}R\left(\ln \frac {t_1}{s_1},0\right)
-2t_1^{H_1}s_1^{H_1}s_2^{2H_2}R\left(\ln \frac {t_1}{s_1},0\right)\]
\[-2s_1^{2H_1}t_2^{H_2}s_2^{H_2}R\left(0,\ln \frac{t_2}{s_2}\right)
+2t_1^{H_1}s_1^{H_1}t_2^{H_2}s_2^{H_2}\left(R\left(\ln \frac {t_1}{s_1},\ln \frac{t_2}{s_2}\right)+R\left(\ln \frac{t_1}{s_1},-\ln \frac{t_2}{s_2}\right)\right)\]
\[=t_1^{2H_1}t_2^{2H_2}+t_1^{2H_1}s_2^{2H_2}+s_1^{2H_1}t_2^{2H_2}+s_1^{2H_1}s_2^{2H_2}-(t_1^{2H_1}+s_1^{2H_1})t_2^{H_2}s_2^{H_2}F_{H_2}\left(\ln \frac{t_2}{s_2}\right)\]
\[-t_1^{H_1}s_1^{H_1}(t_2^{2H_2}+s_2^{2H_2})F_{H_1}\left(\ln \frac {t_1}{s_1}\right)+t_1^{H_1}s_1^{H_1}t_2^{H_2}s_2^{H_2}F_{H_1}\left(\ln \frac {t_1}{s_1}\right)F_{H_2}\left(\ln \frac{t_2}{s_2}\right).\]

%\[=t_1^{2H_1}t_2^{2H_2}+t_1^{2H_1}s_2^{2H_2}+s_1^{2H_1}t_2^{2H_2}+s_1^{2H_1}s_2^{2H_2}\]
%\[-(t_1^{2H_1}+s_1^{2H_1})t_2^{H_2}s_2^{H_2}\left(\frac{t_2^{H_2}}{s_2^{H_2}}+\frac{s_2^{H_2}}{t_2^{H_2}}-\left|\sqrt{\frac{t_2}%{s_2}}-\sqrt{\frac{s_2}{t_2}}\right|^{2H_2}\right)\]
%\[-t_1^{H_1}s_1^{H_1}(t_2^{2H_2}+s_2^{2H_2})\left(\frac{t_1^{H_1}}{s_1^{H_1}}+\frac{s_1^{H_1}}{t_1^{H_1}}-\left|\sqrt{\frac{t_1}%{s_1}}-\sqrt{\frac{s_1}{t_1}}\right|^{2H_1}\right)\]
%\[+t_1^{H_1}s_1^{H_1}t_2^{H_2}s_2^{H_2}\left(\frac{t_1^{H_1}}{s_1^{H_1}}+\frac{s_1^{H_1}}{t_1^{H_1}}-\left|\sqrt{\frac{t_1}{s_1}}-\sqrt{\frac{s_1}{t_1}}\right|^{2H_1}\right)\left(\frac{t_2^{H_2}}{s_2^{H_2}}+\frac{s_2^{H_2}}{t_2^{H_2}}-\left|\sqrt{\frac{t_2}{s_2}}-\sqrt{\frac{s_2}{t_2}}\right|^{2H_2}\right)\]
Therefore, from  \eqref{fh} we have
\[ \E\left(\Delta_{\mathbf{s}} X(\mathbf{t})\right)^2=
t_1^{2H_1}t_2^{2H_2}+t_1^{2H_1}s_2^{2H_2}+s_1^{2H_1}t_2^{2H_2}+s_1^{2H_1}s_2^{2H_2}\]
\[-(t_1^{2H_1}+s_1^{2H_1})\left(t_2^{2H_2}+s_2^{2H_2}-\left|t_2-s_2\right|^{2H_2}\right)-(t_2^{2H_2}+s_2^{2H_2})\left(t_1^{2H_2}+s_1^{2H_2}-\left|t_1-s_1\right|^{2H_1}\right)\]
\[+\left(t_1^{2H_1}+s_1^{2H_1}-\left|t_1-s_1\right|^{2H_1}\right)\left(t_2^{2H_2}+s_2^{2H_2}-\left|t_2-s_2\right|^{2H_2}\right)=|t_1-s_1|^{2
H_1}|t_2-s_2|^{2
H_2}.\]

The proposition is proved.
\end{proof}
\section{Theorem for Covariance Function}
\begin{lemma}
Suppose that there exists a covariance function $R:\R^2\rightarrow \R$  such that
	
  \begin{itemize}

  \item[$(i)$] $R$ does not coincide identically with $R_0$.
   \item[$(ii)$]
\begin{equation}\label{sym}\forall \mathbf{v}\in \R^2: R(\mathbf{v})+R(v_1,-v_2)=2R_0(\mathbf{v}).
\end{equation}
Then there exists  Gaussian self-similar random field $\{X(\mathbf{t}),\mathbf{t}\in \R_+^2\}$ with stationary rectangular increments such that
 $\E X(\mathbf{t})X(\mathbf{s})$ does not coincide with\\  $\frac14 \prod_{i=1,2}\left(|t_i|^{2H_i}+|s_i|^{2H_i}-|t_i -
s_i|^{2H_i} \right).$
\end{itemize}
\end{lemma}
   \begin{proof}
   The finite dimensional distributions of Gaussian fields are uniquely determined  by its mean and covariance functions.
   So there exists probability space and zero mean strictly stationary Gaussian random field $\{Y(\mathbf{t}),{\mathbf{t}\in \R^2}\}$ with covariance function $R$.
    We can define a centered Gaussian random field $\{X(\mathbf{t}),{\mathbf{t}\in \R_+^2}\}$ as
    $X(\mathbf{t})=t_1^{H_1}t_1^{H_2}Y(\ln t_1,\ln t_2)$. The rectangular increments of $X$ have zero mean Gaussian distribution. Therefore, it follows from condition (\ref{sym}) and Proposition \ref{rem1} that variance of $\Delta_{(s_1,s_2)} X(t_1,t_2)$ is equal to the variance of $\Delta_{(0,0)} X(t_1-s_1,t_2-s_2)$. Hence, $X$ has stationary rectangular increments. By Proposition 2.1.1. of \cite{taqqu} we have that $X$ is a self-similar field with index $\mathbf{H}=(H_1,H_2).$
    Proof follows from condition $(i)$.
   \end{proof}
\begin{theorem}
Let $R_\theta:\R^2\rightarrow \R^2$ be a function defined by the formula
\begin{equation}
 {R}_\theta(\mathbf{v})=\frac{1}{4}F_{H_1}(v_1)F_{H_2}(v_2)\left(1+\theta e^{-H_1 |v_1|-H_2|v_2|}\sinh \left(H_1 v_1\right)\sinh \left(H_2 v_2\right)\right),
\end{equation}
where $0<H_1<1,~0<H_2<1,$ $\theta\in\R$ be some number. Then
\begin{enumerate}
\item[$(i)$] $\forall \mathbf{v}\in\R^2:  {R}_\theta(\mathbf{v})=R_\theta(-\mathbf{v});$
\item[$(ii)$] $ {R}_\theta$ does not coincide identically with $R_0;$ 
\item[$(iii)$] $\forall \mathbf{v}\in\R^2: {R}_\theta(\mathbf{v})+ {R}_\theta(v_1,-v_2)= {R}_\theta(\mathbf{v})+ {R}_\theta(-v_1,v_2)=2R_0(\mathbf{v});$
\item[$(iv)$] for any $0<H_1<1,~0<H_2<1$ there exists such $\theta\in\R$ that $ {R}_\theta (\mathbf{u}-\mathbf{v})$ is a positive definite function on $\R^4$.
\end{enumerate}
\end{theorem}
\begin{proof}
Statements $(i)-(iii)$ are trivial. So we prove only statement $(iv)$. Recall that any function $f:\R^2_+\rightarrow\R$
 that is a Fourier transform of a positive integrable function, is positive definite. Therefore, to establish that
  $R_\theta$ is a positive definite function on $\R^2$, it is sufficient to prove that  its Fourier inverse transform is a positive function. In this connection, consider this Fourier transform
\begin{equation}
\label{fur}
\int_{\R^2}e^{2i\pi(xv_1+yv_2)}R_\theta(\mathbf{v})d\mathbf{v}.
\end{equation}  
Note that
\[\int_{\R^2}\left|R_\theta(\mathbf{v})\right|d\mathbf{v}\leq \int_{\R^2}R_0(\mathbf{v})\left(1+|\theta| e^{-H_1 |v_1|-H_2|v_2|}|\sinh \left(H_1 v_1\right)||\sinh \left(H_2 v_2\right)|\right)d\mathbf{v} \]
\[= \int_{\R^2}R_0(\mathbf{v})\left(1+\frac{|\theta|}{4} \left(1-e^{-2H_1 |v_1|}\right)\left(1-e^{-2H_2 |v_2|}\right)\right)d\mathbf{v}\leq \left(1+\frac{|\theta|}{4}\right)\int_{\R^2}R_0(\mathbf{v})d\mathbf{v}\]
\[=\left(\frac{1}{4}+\frac{|\theta|}{16}\right)\int_{\R}F_{H_1}(v_1)dv_1\int_{\R}F_{H_2}(v_2)dv_2.\]
Furthermore, for any $0<H<1$ we have
\[\int_{\R}F_H(v)dv=2\int_{\R_+}F_H(v)dv=2\int_{\R_+}\left(e^{Hv}+e^{-Hv}-e^{Hv}\left(1-e^{-v}\right)^{2H}\right)dv\]
\[\leq 2\int_{\R_+}\left(e^{Hv}+e^{-Hv}-e^{Hv}\left(1-e^{-v}\right)^{2}\right)dv= 2\int_{\R_+}\left(e^{-Hv}+2e^{-(1-H)v}-e^{-(2-H)v}\right)dv\]
\[=\frac{2}{H}+\frac{2(3-H)}{(1-H)(2-H)}.\]
It means that Fourier transform from (\ref{fur}) is defined correctly.
Further, taking into account equalities  \[R_\theta(\mathbf{v})+R_\theta(v_1,-v_2)=\frac12F_{H_1}(v_1)F_{H_2}(v_2)\] and 
\[R_\theta(\mathbf{v})-R_\theta(v_1,-v_2)=\frac{1}{2}\theta F_{H_1}(v_1)F_{H_2}(v_2) e^{-H_1 |v_1|-H_2|v_2|}\sinh \left(H_1 v_1\right)\sinh \left(H_2 v_2\right),\] we get the following relations
\[\int_{\R^2}e^{2i\pi(xv_1+yv_2)}R_\theta(\mathbf{v})d\mathbf{v}=\int_{\R_+\times\R}e^{2i\pi(xv_1+yv_2)}R_\theta(\mathbf{v})d\mathbf{v}+\int_{\R_+\times\R}e^{2i\pi (-xv_1-yv_2)}R_\theta(-\mathbf{v})d\mathbf{v}\]
\[=2\int_{\R_+\times\R}\cos\left(2\pi (xv_1+yv_2)\right)R_\theta(\mathbf{v})d\mathbf{v}\]
\[=2\int_{\R_+\times\R}\left(\cos(2\pi xv_1)\cos(2\pi yv_2)-\sin(2\pi xv_1)\sin(2\pi yv_2)\right)R_\theta(\mathbf{v})d\mathbf{v}\]
\[=2\int_{\R_+\times\R_+}\cos(2\pi xv_1)\cos(2\pi yv_2)\left(R_\theta(\mathbf{v})+R_\theta(v_1,-v_2)\right)d\mathbf{v}\]
\[-2\int_{\R_+\times\R_+}\sin(2\pi xv_1)\sin(2\pi yv_2)\left(R_\theta(\mathbf{v})-R_\theta(v_1,-v_2)\right)d\mathbf{v}\]

\[=\int_{\R_+}\cos(2\pi xv_1)F_{H_1}(v_1)dv_1\int_{\R_+}\cos(2\pi yv_2)F_{H_2}(v_2)dv_2\]
\[-\theta\int_{\R^2_+}\sin(2\pi xv_1)\sin(2\pi yv_2)F_{H_1}(v_1)F_{H_2}(v_2)e^{-H_1 |v_1|-H_2|v_2|}\sinh \left(H_1 v_1\right)\sinh \left(H_2 v_2\right)d\mathbf{v}\]
\[=\int_{\R_+}F_{H_1}(v_1)\cos(2\pi xv_1)dv_1\int_{\R_+}F_{H_2}(v_2)\cos(2\pi yv_2)dv_2\]
\begin{equation}
\label{fur_2}
-\frac{\theta}{4}\int_{\R_+}F_{H_1}(v_1)\left(1-e^{-2H_1v_1}\right)\sin(2\pi xv_1)dv_1\int_{\R_+}F_{H_2}(v_2)\left(1-e^{-2H_2v_2}\right)\sin(2\pi yv_2)dv_2.
\end{equation}
Now our aim is to prove that the right hand side of (\ref{fur_2}) is non-negative for all $\mathbf{H} \in (0,1)^2$ and for all $x,y\in \R.$
Note that $F_H(v)$ is positive definite as a covariance function, therefore the integral $\int_{\R_+}F_{H}(v)\cos(2\pi xv)dv$ is positive for any $H\in (0,1)$ and $x\in \R$.

Evidently, it is sufficient to establish that there exist $\theta \in \R$ such that for any $H\in (0,1)$ and $x\in \R$

\begin{equation}
\label{main_ineq}
\int_{\R_+}F_{H}(v)\cos( xv)dv>\frac{\sqrt{|\theta}|}{2}\left|\int_{\R_+}F_{H}(v)\left(1-e^{-2Hv}\right)\sin(xv)dv\right|.
\end{equation}
Denote integrals in the left and right hand sides of inequality (\ref{main_ineq}) as $a(x)$ and $b(x)$ respectively.
It is sufficient to establish (\ref{main_ineq}) when $x>0,$ because $b(x)$ is an odd function and $a(x)$ is an even one.

Recall that for $|x|\leq 1$ \[(1+x)^{2H}=1+\sum_{n=1}^{\infty}\binom{2H}{ n}x^n, \quad \text{and} \quad \sum_{n=1}^{\infty}\binom{2H}{ n}(-1)^{n-1}=1.\]
where binomial coefficients equal
\[\binom{2H}{n}=\frac{2H}{1}\frac{(2H-1)}{2}\frac{(2H-2)}{3}\frac{(2H-3)}{4}\cdots\frac{(2H-n+1)}{n}.\]
It's obvious that $(2H-k)<0$ for $k\ge 2.$ So the binomial coefficients have the following properties
\[\binom{2H}{n}=(-1)^{n-1}\left|\binom{2H}{n}\right|,\quad 0<H\leq\frac12, n\geq 1.\]
\[\binom{2H}{n}=(-1)^{n-2}\left|\binom{2H}{n}\right|,\quad \frac12\leq H<1,n\geq 2.\]

Then we expand function $F_H$ as
\[F_H(v)=e^{H v}+e^{-H v}-e^{H v} \left(1-e^{-v}\right)^{2H}=e^{H v}+e^{-H v}-e^{H v} \left(1+\sum_{n=1}^{\infty}\binom{2H}{ n} (-1)^n e^{-nv}\right)\]
\[=e^{-H v}- \sum_{n=1}^{\infty}\binom{2H}{ n} (-1)^n e^{-(n-H)v}.\]
And
\[F_H(v)(1-e^{- 2H v})=e^{-H v}-e^{-3 H v}- \sum_{n=1}^{\infty}\binom{2H}{ n} (-1)^n e^{-(n-H)v}+\sum_{n=1}^{\infty}\binom{2H}{ n} (-1)^n e^{-(n+H)v}.\]

For the sequences of functions  $c_n(v)=\binom{2H}{ n} (-1)^n e^{-n v}\cos(xv)dv,n\geq 1$ and\\ $s_n(v)=\binom{2H}{ n} (-1)^n e^{-(n-H)v}\sin(xv)dv,n\geq 1$ the series $\sum_{n=1}^{\infty}c_n(v)$ and $\sum_{n=1}^{\infty}s_n(v)$ converge uniformly, because  $|c_n(v)|\leq\left|\binom{2H} {n}\right|$, $|s_n(v)|\leq\left|\binom{2H}{ n}\right|$ and $\sum_{n=1}^{\infty}\left|\binom{2H}{ n}\right|<+\infty$ (Weierstrass M-test).  
The uniform convergence implies that 
\[a(x)=\int_{\R_+}F_{H}(v)\cos(xv)dv=\int_{\R_+}e^{-H v}\cos(xv)dv \]
\[-\sum_{n=1}^{\infty}\binom{2H}{ n} (-1)^n \int_{\R_+}e^{-(n-H)v}\cos(xv)dv,\]
and
\[b(x)=\int_{\R_+}F_{H}(v)\left(1-e^{-2Hv}\right)\sin( xv)dv=\int_{\R_+}e^{-H v}\sin(xv)dv-\int_{\R_+}e^{-3 H v}\sin(xv)dv \]
\[-\sum_{n=1}^{\infty}\binom{2H}{ n} (-1)^n \int_{\R_+}e^{-(n-H)v}\sin(xv)dv+\sum_{n=1}^{\infty}\binom{2H}{ n} (-1)^n\int_{\R_+} e^{-(n+H)v}\sin(xv)dv.\]

Recall that in the case $\alpha<0$ we have
\[\int_{\R_+} e^{\alpha x}\sin \beta x dx =\frac{\beta}{\alpha^2+\beta^2} \quad \text{and} \quad \int_{\R_+} e^{\alpha x}\cos \beta x dx =\frac{-\alpha}{\alpha^2+\beta^2}.\]
Then
\[a(x)= \frac{H}{H^2+x^2}-\sum_{n=1}^{\infty}\binom{2 H}{n}(-1)^n\frac{n-H}{(n-H)^2+x^2}\]
and
\[b(x)= \frac{2\pi x}{H^2+(2 \pi x)^2}-\frac{2\pi x}{9 H^2+x^2}\]
\[-\sum_{n=1}^{\infty}\binom{2 H}{n}(-1)^n\frac{2\pi x}{(n-H)^2+x^2}+\sum_{n=1}^{\infty}\binom{2 H}{n}(-1)^n\frac{x}{(n+H)^2+x^2}.\]
%And
%$$a(x)-\frac{1-H}{2} b(x)=\frac{H}{H^2+(2 \pi x)^2}-\frac{8 H^2 (1-H)(2\pi x)}{(H^2+(2 \pi x)^2)(9 H^2+(2\pi x)^2)}+$$
%$$-\sum_{n=1}^{+\infty}\binom{2 H}{n}(-1)^n\left(\frac{n-H}{(n-H)^2+(2\pi x)^2}-\frac{4nH(1-H) (2\pi x)}{((n-H)^2+(2\pi x)^2)((n%+H)^2+(2\pi x)^2)}\right)=$$
%$$=\frac{H}{H^2+(2 \pi x)^2}\left(1-\frac{8 H(1-H)(2\pi x)}{9 H^2+(2\pi x)^2}\right)-$$
%$$-\sum_{n=1}^{+\infty}\binom{2 H}{n}(-1)^n\frac{n-H}{(n-H)^2+(2\pi x)^2}\left(1-\frac{2nH(1-H)(2\pi x)}{(n-H)((n+H)^2+(2\pi %x)^2)}\right).$$

Now we verify that $b(x)>0,x>0.$ Consider function $b(x)$ for the both cases $0<H\leq\frac12$ and $\frac12<H<1$.

In the case $0<H\leq\frac12$ we have
\[b\left(x\right)=\frac{8 x H^2}{(H^2+x^2)(9 H^2+ x^2)}+\sum_{n=1}^{\infty}\left|\binom{2 H}{n}\right|\frac{4x n H}{((n-H)^2+x^2)((n+H)^2+x^2)}>0,x>0.\]
In the case $\frac12<H<1$ we have
\begin{equation}
\label{bx}
\begin{gathered}
b\left(x\right)=\frac{8 x H^2}{(H^2+x^2)(9 H^2+ x^2)}+2H\frac{4x H}{((1-H)^2+x^2)((1+H)^2+x^2)}\\
-\sum_{n=2}^{\infty}\left|\binom{2 H}{n}\right|\frac{4x n H}{((n-H)^2+x^2)((n+H)^2+x^2)}
\end{gathered}
\end{equation}
\[\geq\frac{8 x H^2}{(H^2+x^2)(9 H^2+ x^2)}+2H\frac{4x H}{((1-H)^2+x^2)((1+H)^2+x^2)}\]
\[-\sum_{n=2}^{\infty}\left|\binom{2 H}{n}\right|\frac{4x n H}{((1-H)^2+x^2)((1+H)^2+x^2)}\]
\[\geq\frac{8 x H^2}{(H^2+x^2)(9 H^2+ x^2)}+\frac{4x H}{((1-H)^2+x^2)((1+H)^2+x^2)}\left(2H-\sum_{n=2}^{\infty}\left|\binom{2 H}{n}\right|n\right).\]
%Calculate the sum of the series in the last term. 
Note that $2H\binom{2H-1}{n-1}=n\binom{2H}{n}, n\geq 2.$
Since $2H-1>0,$ we see that the following series converges when $|x|\leq 1:$
\[(1+x)^{2H-1}=1+\sum_{n=1}^{\infty}\binom{2H-1}{n}x^n=1+\frac{1}{2H}\sum_{n=2}^{\infty}\binom{2H-1}{n-1}x^{n-1}=1+\frac{1}{2H}\sum_{n=2}^{\infty}\binom{2H}{n}nx^{n-1}.\]

Therefore at point $x=-1$ we have
\[0=2H(1-1)^{2H-1}=2H+\sum_{n=2}^{\infty}\binom{2H}{n}n{(-1)}^{n-1}=2H-\sum_{n=2}^{\infty}\left|\binom{2H}{n}\right|n.\]
Hence, we prove that for any $0<H<1$ 
\[b\left(x\right)>\frac{8 x H^2}{(H^2+x^2)(9 H^2+ x^2)}>0,x>0.\]

That's why inequality (\ref{main_ineq}) follows from $a(x)-\frac{\sqrt{|\theta}|}{2} b(x)>0,x>0.$
Consider the last inequality for the cases $0<H\leq\frac12$ and $\frac12<H<1$ separately.

In the case $0<H\leq\frac12$ assume that $0<|\theta|<(1-H)^2.$ Then
\[a(x)-\frac{\sqrt{|\theta|}}{2} b(x)> a(x)-\frac{1-H}{2} b(x)
=\frac{H}{H^2+x^2}\left(1-\frac{8 H(1-H)x}{9 H^2+x^2}\right)\]
\[+\sum_{n=1}^{+\infty}\left|\binom{2H}{n}\right|\frac{n-H}{(n-H)^2+x^2}\left(1-\frac{2nH(1-H) (2\pi x)}{(n-H)((n+H)^2+x^2)}\right).\]
Note that
\[\min_{x\in(0,+\infty)}\left(1-\frac{8 H(1-H) x}{9 H^2+x^2}\right)=\frac{2H+1}{3},\quad x_{min}=3H,\]
and
\[\min_{x\in(0,+\infty)}\left(1-\frac{2nH(1-H)x}{(n-H)((n+H)^2+x^2)}\right)=1-\frac{nH(1-H)}{n^2-H^2}>\frac{1}{1+H},\quad x_{min}=n+H.\]
Therefore,
\[a(x)-\frac{1-H}{2} b(x)\ge\frac{H}{H^2+x^2}\frac{2H+1}{3}+\sum_{n=1}^{+\infty}\left|\binom{2H}{n}\right|\frac{n-H}{(n-H)^2+x^2}\frac{1}{1+H}>0.\]

Hence, for $H\in(0,\frac 12]$ and $|\theta|<(1-H)^2$ the inequality (\ref{main_ineq}) is true.

Consider the case $\frac12<H<1.$ We find  a lower estimate for the function $a(x),x>0$ and the upper estimate for the function $b(x),x>0.$ 

The function $a(x),x>0$ has the following integral representation.
\[a\left(x\right)=\int_{\R_+}F_{H}(v)\cos(xv)dv=\frac{H}{H^2+x^2}+\frac{I(x)+I(-x)}{2}, x \in \R_+.\]
where
\[I(x)=\int_{\R_+}e^{ixv}\left(e^{vH}-e^{vH}\left(1-e^{-v}\right)^{2H}\right)dv, x\in \R.\]
We reduce the integral $I(x)$ to a tabulated form. Namely,
\[I(x)=\int_{\R_+}e^{(H+ix)v}\left(1-\left(1-e^{-v}\right)^{2H}\right)dv\]
\[=\frac{e^{(H+ix)v}}{H+ix}\left.\left(1-\left(1-e^{-v}\right)^{2H}\right)\right|^{+\infty}_{0}+\frac{2H}{H+ix}\int_{\R_+}e^{(H+ix)v}\left(1-e^{-v}\right)^{2H-1}e^{-v}dv\]
%\[=\left|\begin{array}{c}e^{-v}=t\\
%e^{-v}dv=-dt\end{array}\right|=-\frac{1}{H+ix}+\frac{2H}{H+ix}\int_{0}^{1}t^{-H-ix}\left(1-t\right)^{2H-1}dt\]
\[=-\frac{1}{H+ix}+\frac{2H}{H+ix}B(2H,1-H-ix)=-\frac{1}{H+ix}+\frac{2H}{H+ix}\frac{\Gamma(2H)\Gamma(1-H-ix)}{\Gamma(1+H-ix)},\]
where $\Gamma$ is the gamma function and $B$ is the beta function, defined as $B(z_1,z_2)=\frac{\Gamma(z_1) \Gamma(z_2)}{\Gamma(z_1+z_2)}$, for  $Re(z_1)>0, Re(z_2)>0.$

Recall the basic properties of the gamma function (see \cite{adam}):
\[\Gamma(1+z)=z \Gamma(z), z \in \mathbb{C},\quad \Gamma(1-z)\Gamma(z) = \frac{\pi}{\sin{\pi z}}, z \in \mathbb{C} \quad \mbox{ and } \Gamma(z)\Gamma(\bar{z})=\left|\Gamma(z)\right|^2,z\in \mathbb{C},\] where
\[\sin(u+i v)= \sin u \cosh v + i \cos u \sinh v, \quad u,v \in \R.\]

Applying these properties, we get
\[I(x)=-\frac{1}{H+ix}+\frac{\Gamma(1+2H)}{H^2+x^2}\frac{\Gamma(H+ix)\Gamma(1-H-ix)}{\Gamma(H+ix)\Gamma(H-ix)}\]
\[=-\frac{1}{H+ix}+\frac{\Gamma(1+2H)}{(H^2+x^2)|\Gamma(H+ix)|^2}\left(\frac{\pi}{\sin (\pi (H+ix))}\right)\]
\[=-\frac{1}{H+ix}+\frac{\Gamma(1+2H)}{(H^2+x^2)|\Gamma(H+ix)|^2}\left(\frac{\pi \sin (\pi (H-ix))}{|\sin (\pi (H+ix))|^2}\right)\]
\[=-\frac{1}{H+ix}+\frac{\pi\Gamma(1+2H)}{(H^2+x^2)|\Gamma(H+ix)|^2}\left(\frac{ \sin (\pi H) \cosh (\pi x) - i\cos (\pi H) \sinh (\pi x)}{\sin^2 (\pi H) \cosh^2 (\pi x) + \cos^2 (\pi H) \sinh^2 (\pi x)}\right),x\in \R.\]
Therefore,
\[a\left(x\right)=\frac{H}{H^2+x^2}+\frac{I(x)+I(-x)}{2}=\frac{H}{H^2+x^2}+\frac{1}{2}\left(-\frac{1}{H+ix}-\frac{1}{H-ix}\right)\]
\[+\frac{\pi\Gamma(1+2H)}{(H^2+x^2)|\Gamma(H+ix)|^2}\left(\frac{ \sin (\pi H) \cosh (\pi x)}{\sin^2 (\pi H) \cosh^2 (\pi x) + \cos^2 (\pi H) \sinh^2 (\pi x)}\right)\]

\[=\frac{\pi\Gamma(1+2H)}{(H^2+x^2)|\Gamma(H+ix)|^2}\left(\frac{\sin(\pi H)\cosh(\pi x)}{\cosh^2(\pi x)-\cos(2H\pi)}\right).\]
Using the formula 6.1.25 in \cite{adam} for absolute value of the gamma function we prove
the following inequality.
%$$|\Gamma(H+ix)|\leq \sqrt{2\pi} |H+ix|^{H-1/2}e^{-\pi x/2}\exp\left\{\frac{1}{6|H+ix|}\right\},x>0.$$
\[\frac{\Gamma^2(H)}{|\Gamma(H+ix)|^2}=\prod_{n=0}^{+\infty}\left(1+\frac{x}{(n+H)^2}\right)\geq\prod_{n=0}^{+\infty}\left(1+\frac{x}{(n+1)^2}\right)\]
\[=\frac{\Gamma^2(1)}{|\Gamma(1+ix)|^2}=\frac{\sinh (\pi x)}{\pi x}.\]

Therefore, %the case $1/2<H<1$
\[a(x)\geq \frac{\pi\Gamma(1+2H)}{H^2+x^2}\left(\frac{\sin(\pi H)\cosh(\pi x)}{\cosh^2(\pi x)-\cos(2H\pi)}\right)\frac{\sinh (\pi x)}{\pi x \Gamma^2(H)}\]
\[=\frac{\Gamma(1+2H)\sin(\pi H)}{\Gamma^2(H)(H^2+x^2)x}\tanh(\pi x)\left(\frac{\cosh^2(\pi x)}{\cosh^2(\pi x)-2\cos(2H\pi)}\right)\]
\[\geq\frac{\Gamma(1+2H)\sin(\pi H)}{2\Gamma^2(H)(H^2+x^2)x}\tanh(\pi x).\]
Note that $\frac{\tanh(\pi H)x}{x},x>0$ is a decreasing function and $\tanh(\pi x),x>0$ is an increasing one.
Hence,
\begin{align}
\label{lower_a}
a(x)\geq & \frac{\Gamma(1+2H)\sin(\pi H)}{2\Gamma^2(H)(H^2+x^2)x}\tanh(\pi H),& x \geq H,\\
a(x)\geq & \frac{\Gamma(1+2H)\sin(\pi H)}{2\Gamma^2(H)(H^2+x^2)H}\tanh(\pi H),& 0\leq x<H.
\end{align}

Return to upper estimate for $b(x).$ In the case $\frac12<H<1$ it follows from (\ref{bx}) that
\[b(x)\leq \frac{8 x H^2}{(H^2+x^2)(9 H^2+ x^2)}+\frac{8 x H^2 }{((1-H)^2+x^2)((1+H)^2+x^2)}\]
\[\leq \frac{8 H^2}{H^2+x^2}\left(\frac{x}{9H^2+x^2}+\frac{x}{(1-H)^2+x^2}\right).\]

Note that
\[
\frac{x}{9H^2+x^2}+\frac{x}{(1-H)^2+x^2}\leq \frac{1}{x}+\frac{1}{x}=\frac{2}{x}, x>0, \]
and  for $0\leq x<H$
\[\frac{x}{9H^2+x^2}+\frac{x}{(1-H)^2+x^2}\leq \left(\frac{1}{10H}+\frac{1}{2(1-H)}\right)=\frac{1+4H}{10H(1-H)}\leq \frac{1}{2H(1-H)}. \]

Therefore,
\begin{equation}
\label{uper_b_1}
b(x)\leq   \frac{16 H^2}{(H^2+x^2)x}, x \geq H,\\
\end{equation}
and
\begin{equation}
\label{uper_b_2}
b(x) \leq \frac{1}{H^2+x^2}\frac{4H}{1-H},  0\leq x<H.
\end{equation}

Thus, from \eqref{lower_a}, \eqref{uper_b_1}, and \eqref{uper_b_2} we get the following inequalities:
\[a(x)-\frac{\sqrt{|\theta|}}{2} b(x)>\frac{\Gamma(1+2H)\sin(\pi H)}{2\Gamma^2(H)(H^2+x^2)x}\tanh(\pi H)-\frac{\sqrt{|\theta|}}{2} \frac{8 H^2}{H^2+x^2}\left(\frac{2}{x}\right)\]
\[=\frac{1}{(H^2+x^2)x}\left(\frac{\Gamma(1+2H)\sin(\pi H)}{2\Gamma^2(H)}\tanh(\pi H)-\sqrt{|\theta|}8 H^2\right),x \geq H,\]
and
\[a(x)-\frac{\sqrt{|\theta|}}{2} b(x)>\frac{\Gamma(1+2H)\sin(\pi H)}{2\Gamma^2(H)(H^2+x^2)H}\tanh(\pi H)-\frac{\sqrt{|\theta|}}{2} \frac{1}{H^2+x^2}\frac{4H}{1-H}\]
\[=\frac{1}{(H^2+x^2)H}\left(\frac{\Gamma(1+2H)\sin(\pi H)}{2\Gamma^2(H)}\tanh(\pi H)-\sqrt{|\theta|}\frac{2H^2}{1-H}\right),0\leq x<H.
\]
Thus, in  the case $\frac{1}{2}<H<1$ we prove that $a(x)-\frac{\sqrt{|\theta|}}{2}b(x)>0,x\geq 0$ for
\begin{equation}
\label{eq:last_theta}
\sqrt{|\theta|}<\frac{\Gamma(1+2H)\sin(\pi H)}{2\Gamma^2(H)}\tanh(\pi H)\frac{1-H}{4H^2}.
\end{equation}
Note that according to the tables of values for gamma function $\Gamma (x)\geq 0.88, x >0.$ Therefore, for $0<H<1$
\[\frac{\Gamma(1+2H)}{ 8 H^2 \Gamma^2(H)}=\frac{\Gamma (H)\Gamma(H+\frac 12) 2^{2H-\frac12}}{ 4 H \sqrt{2\pi}\Gamma^2(H)}=
\frac{\Gamma(H+\frac 12)}{4^{1-H} 2  \sqrt{ \pi}\Gamma(H+1)}\leq \frac{\Gamma(\frac 12)}{4^{1-H} 2 \sqrt{\pi} 0.88}<1.\] 
Recall that $0 \leq \tanh (\pi H)\leq 1, H>0.$ Then the right hand side of \eqref{eq:last_theta} is less then $1-H.$  

Finally, summarising the both cases, we have that the inequality \eqref{main_ineq}  is true if 
\[\sqrt{|\theta|}<\min\left\{1-H,\frac{\Gamma(2H)}{\Gamma^2(H)}\frac{1-H}{4H}\sin(\pi H)\tanh(\pi H)\right\}=\frac{\Gamma(2H)}{\Gamma^2(H)}\frac{1-H}{4H}\sin(\pi H)\tanh(\pi H).\]

Theorem is proved.
\end{proof}


\begin{thebibliography}{99}
\bibitem{adam}  Abramowitz M. and Stegun I. A.: Handbook of Mathematical Functions with Formulas, Graphs, and Mathematical Tables. \emph{National Bureau of Standards Applied Mathematics} Series 55. Tenth Printing. 1972.

\bibitem{env3} Benson D., Meerschaert M. M., B\"{a}umer B., Scheffler H.-P.: 
Aquifer operator scaling and the effect on solute mixing and
dispersion. \emph{Water Resour. Res.} \textbf{42}, (2006), 1 -- 18.

\bibitem{alp} 
Ayache A., Leger S., Pontier M.: Drap brownien fractionnaire. \emph{Potential Analysis} \textbf{17.1}, (2002), 31 -- 43.

\bibitem{env4}
Bonami A. and Estrade A.:  Anisotropic analysis of some
Gaussion models. \emph{J. Fourier Anal. Appl.} \textbf{9}, (2003), 215 -- 236.

\bibitem{env5}
Davies S. and Hall P.: Fractal analysis of surface roughness
by using spatial data (with discussion). \emph{J. Roy. Statist.
Soc. Ser. B} \textbf{61}, (1999), 3 -- 37.

\bibitem{maej} Embrechts P. and Maejima M.: Selfsimilar Processes, \emph{Princeton
University Press}. 2001.

\bibitem{taqqu}
Genton M.G., Perrin O., Taqqu M. S.: Self-Similarity
and Lamperti Transformation for Random Fields, \emph{Stochastic
Models}. \textbf{23}, (2007), 397 -- 411, 


\bibitem{env1} Lu W.Z. and Wang X.K.: Evolving trend and self-similarity of ozone
pollution in central Hong Kong ambient during 1984 -- 2002. \emph{The
Science of the Total Environment}. \textbf{357}, (2006), 160 -- 168.

\bibitem{env2} Morata A., Martin M.L., Luna M.Y., Valero F.: Self-similarity
patterns of precipitation in the Iberian Peninsula. \emph{Journal
of Theoretical and Applied Climatology}, \textbf{85},  (2006), 41 -- 59.

\bibitem{Sam} Samorodnitsky G. and Taqqu M.S.: Stable
Non-Gaussian Random Processes: \emph{Stochastic Models with
Infinite Variance}. Chapman \& Hall. 1994.

\bibitem{taqqu81}
Taqqu M.S.: Selfsimilar processes and related ultraviolet and infrared
catastrophes, in: Random Fields: Rigorous Results in Statistical Mechanics and
Quantum Field Theory, \emph{Colloquia Mathematica Societatis Janos Bolya.} \textbf{27.2}, (1981), 1027 -- 1096.  

\end{thebibliography}
\end{document}